\documentclass[a4paper,12pt]{amsart}
\usepackage{amsfonts,amsmath,amsthm,amssymb}
\usepackage{enumitem}
\usepackage[margin=1in]{geometry}
\usepackage[colorlinks=true,linkcolor=blue,filecolor=blue,citecolor=red]{hyperref}
\usepackage{cleveref}

\theoremstyle{plain}
\newtheorem{theorem}{Theorem}[section]
\newtheorem{corollary}[theorem]{Corollary}
\newtheorem{lemma}[theorem]{Lemma}

\theoremstyle{definition}

\begin{document}

\title[A note on the squares of the form\ldots]{A note on the squares of the form $\prod_{k=1}^n (2k^2+l)$ with $l$ odd}
\author[Russelle Guadalupe]{Russelle Guadalupe}
\address{Institute of Mathematics, University of the Philippines-Diliman\\
	Quezon City 1101, Philippines}
\email{rguadalupe@math.upd.edu.ph}

\renewcommand{\thefootnote}{}



\renewcommand{\thefootnote}{\arabic{footnote}}
\setcounter{footnote}{0}

\begin{abstract}
	Let $l$ be a positive odd integer. Using Cilleruelo's method, we establish an explicit lower bound $N_l$ depending on $l$ such that for all $n\geq N_l$, $\prod_{k=1}^n (2k^2+l)$ is not a square. As an application, we determine all values of $n$ such that $\prod_{k=1}^n (2k^2+l)$ is a square for certain values of $l$.
\end{abstract}

\maketitle

\section{Introduction}

For positive integers $a,c$ and $n$ with $\gcd(a,c)=1$, define $P_{a,c}(n) := \prod_{k=1}^n (ak^2+c)$. The question of determining whether the sequence $\{P_{a,c}(n)\}_{n\geq 1}$ contains infinitely many squares has been extensively studied. In 2008, Amdeberhan, Medina, and Moll \cite{amdeber} conjectured that $P_{1,1}(n)$ is not a square for $n\geq 4$ and $P_{4,1}(n)$ is not a square for $n\geq 1$. In addition, they computationally verified that $P_{1,1}(n)$ is not a square for $n \leq 10^{3200}$. Shortly, Cilleruelo \cite{cille} proved the conjecture for $P_{1,1}(n)$ and showed that $P_{1,1}(n)$ is square only for $n=3$. Fang \cite{fang} settled the conjecture for $P_{4,1}(n)$. Yang, Togb\'{e}, and He \cite{yangtog} found all positive integer solutions to the equation $P_{a,c}(n) = y^l$ for coprime integers $a, c\in \{1,\ldots, 20\}$ and $l\geq 2$. Yin, Tan, and Luo \cite{yintan} obtained the $p$-adic valuation of $P_{1,21}(n)$ for all primes $p$ and showed that $P_{1,21}(n)$ is not a square for all $n\geq 1$. Chen, Wang and Hu \cite{chenwh} proved that $P_{1,23}(n)$ is a square only for $n=3$.\\

Ho \cite{ptungho} studied the sequence $\{P_{1,m^2}(n)\}_{n\geq 1}$ for $m\geq 1$ and proved that if $m$ has divisors of the form $4q+1$ and $N=\max\{m,10^8\}$, then $P_{1,m^2}(n)$ is not a square for all $n\geq N$. Recently, Zhang and Niu \cite{zhangniu} generalized Ho's result by showing that there is a positive integer $N_q$ depending on $q$ such that $P_{1,q}(n)$ is not a square for all $n\geq N_q$. Motivated by their work, we consider the problem of finding squares of the form $P_{2,l}(n)$ for a positive odd integer $l$. In this paper, we apply Cilleruelo's technique to prove the following result.

\begin{theorem}
\label{th:thm1}
Let $l$ be a positive odd integer. Then there exists a positive integer $N_l$ depending on $l$ such that for all $n\geq N_l$, $P_{2,l}(n)=\prod_{k=1}^n (2k^2+l)$ is not a square. 
\end{theorem}

As an application of Theorem \ref{th:thm1}, we show that for certain values of $l$, there are finitely many squares of the form $P_{2,l}(n)$ for some integer $n\geq 1$. 

\begin{corollary}
\label{co:cor2}
$P_{2,1}(n)=\prod_{k=1}^n (2k^2+1)$ is not a square for all $n\geq 1$.
\end{corollary}

\begin{corollary}
\label{co:cor3}
$P_{2,3}(n)=\prod_{k=1}^n (2k^2+3)$ is not a square for all $n\geq 1$.
\end{corollary}

\begin{corollary}
\label{co:cor4}
$P_{2,7}(n)=\prod_{k=1}^n (2k^2+7)$ is a square only for $n=1$.
\end{corollary}

We organize the paper as follows. In Section 2, we present several preliminary lemmas needed for the proof of Theorem \ref{th:thm1}. In Sections 3 and 4, we prove Theorem \ref{th:thm1} by working on two separate cases: $l\geq 3$ and $l=1$. In Sections 5, 6 and 7, we prove Corollaries \ref{co:cor2}, \ref{co:cor3} and \ref{co:cor4} using Theorem \ref{th:thm1}. 

\section{Preliminaries}

In this section, we list some preliminary results that are used in the proofs of our main theorem. Throughout this section, we denote a prime by $p$ and a positive odd integer by $l$.

\begin{lemma}
\label{le:lem4} Let $n > \sqrt{l/2}$ be such that $P_{2,l}(n)$ is a square and let $p$ be a prime divisor of $P_{2,l}(n)$. Then $p < 2n$.
\end{lemma}

\begin{proof}
Since $P_{2,l}(n)$ is a square and $p\mid P_{2,l}(n)$, we have $p^2\mid P_{2,l}(n)$. We consider two cases:
\begin{enumerate}
	\item Suppose $p^2\mid 2k^2+l$ for some $1\leq k\leq n$. Then $p\leq \sqrt{2k^2+l}\leq \sqrt{2n^2+l} < 2n$ since $l < 2n^2$. 
	\item Suppose $p\mid 2k^2+l$ and $p\mid 2l^2+1$ for all $1\leq k< l\leq n$. Then $p\mid 2(l^2-k^2) = 2(l-k)(l+k)$. Since $l$ is odd, $P_{2,l}(n)$ is odd, so $p\nmid 2$. Thus, we see that
	either $p\mid l-k$ or $p\mid l+k$, which implies that $p\leq \max\{l-k,l+k\} < 2n$.
\end{enumerate}
\end{proof}

The above lemma implies that if $n > \sqrt{l/2}$, then $P_{2,l}(n) = \prod_{p < 2n}p^{\alpha_p}$, where $p$ is odd and
\begin{align}
\label{eq:eq1}
\alpha_p = \sum_{j\leq \log(2k^2+l)/\log p} \#\{1\leq k\leq n: p^j\mid 2k^2+l\}.
\end{align}
In particular, we have $\alpha_2=0$. Now, observe that for $1\leq k\leq n$, 
\[2k^2+l = k^{\log(2k^2+l)/\log k} > k^{\log(2n^2+l)/\log n}\]
where the last inequality follows from the fact that for $l\geq 1$, the function $f_l(x) := \log(2x^2+l)/\log x$ is decreasing on $(1,+\infty)$. Setting $\lambda = \log(2n^2+l)/\log n$, we have 
$P_{2,l}(n) > (n!)^{\lambda}$ and writing $n! = \prod_{p\leq n} p^{\beta_p}$, where
\begin{align}
\label{eq:eq2}
\beta_p = \sum_{j\leq \log n/\log p} \#\{1\leq k\leq n: p^j\mid k\} = \sum_{j\leq \log n/\log p}\left\lfloor\dfrac{n}{p^j}\right\rfloor, 
\end{align}
we deduce that 
\begin{align}
\label{eq:eq3}
\sum_{p\leq n} \beta_p\log p\leq \dfrac{1}{\lambda}\sum_{p < 2n}\alpha_p\log p.
\end{align}

\begin{lemma}
\label{le:lem5}
Let $p$ be an odd prime with $p\nmid l$ and let $j > 0$. Then the congruence $2x^2\equiv -l\pmod{p^j}$ has $1+\left(\frac{-2l}{p}\right)$ solutions. 
\end{lemma}

\begin{proof}
This follows from rewriting the congruence as $(2x)^2\equiv -2l\pmod{p^j}$ and applying \cite[Thm. 5.1]{lkhua}.
\end{proof}

\begin{lemma}
\label{le:lem6}
Let $p$ be an odd prime with $p\nmid l$.
\begin{enumerate}
	\item If $\left(\frac{-2l}{p}\right) = -1$, then $\alpha_p = 0$.
	\item If $\left(\frac{-2l}{p}\right) = 1$, then $\frac{\alpha_p}{\lambda}-\beta_p\leq \frac{\log(2n^2+l)}{\log p}$.
\end{enumerate}
\end{lemma}

\begin{proof}
\begin{enumerate}
	\item If $\left(\frac{-2l}{p}\right) = -1$, then by Lemma \ref{le:lem5}, the congruence $2x^2+l\equiv 0\pmod {p^j}$ has no solutions for $j > 0$. Thus, $p^j\nmid 2k^2+l$ for all integers $j, k\geq 1$ and $\alpha_p=0$. 
	\item If $\left(\frac{-2l}{p}\right) = 1$, then by Lemma \ref{le:lem5}, the congruence $2x^2+l\equiv 0\pmod {p^j}$ has at most two solutions in an interval of length $p^j$ for each $j > 0$. From (\ref{eq:eq1}), we deduce that
	\begin{align}
	\label{eq:eq4}
	\alpha_p \leq \sum_{j\leq \log(2k^2+l)/\log p} 2\left\lceil\dfrac{n}{p^j}\right\rceil \leq \sum_{j\leq \log(2k^2+l)/\log p} \lambda\left\lceil\dfrac{n}{p^j}\right\rceil
	\end{align}
	Thus, we obtain
	\[\begin{aligned}
	\dfrac{\alpha_p}{\lambda}-\beta_p &\leq \sum_{j\leq \log n/\log p}\left(\left\lceil\dfrac{n}{p^j}\right\rceil-\left\lfloor\dfrac{n}{p^j}\right\rfloor\right)+\sum_{\log n/\log p < j \leq \log(2n^2+l)/\log p} \left\lceil\dfrac{n}{p^j}\right\rceil\\
	&\leq \sum_{j\leq \log n/\log p}1 +\sum_{\log n/\log p < j \leq \log(2n^2+l)/\log p} 1 = \dfrac{\log(2n^2+l)}{\log p}.
	\end{aligned}\]
\end{enumerate}
\end{proof}

\begin{lemma}
\label{le:lem7}
Suppose $p^s$ is the largest power of an odd prime $p$ dividing $l\geq 3$. Then $\alpha_p\leq \gamma_{l,p}(n)$, where
\[\begin{aligned}
\gamma_{l,p}(n) &= \dfrac{n}{2}\left(\dfrac{3p-2-p^{-s}}{(p-1)^2}-\dfrac{s+3}{p^s(p-1)}+\dfrac{4}{p^{s/2}(p-1)}\right)-\dfrac{2np^{s/2}}{(p-1)(2n^2+l)}\\
&+\dfrac{s(s+5)}{4}+2p^{s/2}\left(\dfrac{\log(2n^2+l)}{\log p}-s\right).
\end{aligned}\]
\end{lemma}

\begin{proof}
Write $l=p^sq$, where $p\nmid q$. We consider the number of solutions to the congruence $2x^2 \equiv -p^sq\pmod{p^j}$ with $j > 0$, which is equivalent to $(2x)^2\equiv -2p^sq\pmod{p^j}$. We work on two cases:
\begin{enumerate}
	\item Suppose $s\geq j$. Then $2x\equiv p^{\lceil j/2\rceil}, p^{\lceil j/2\rceil+1},\ldots, p^j\pmod{p^j}$, so there are exactly $j-\lceil j/2\rceil+1$ solutions in this case.
	\item Suppose $s < j$. Then $(2x)^2 = mp^j -2p^sq = p^s(mp^{j-s}-2q)$ for some positive integer $m$. Since $p$ does not divide $mp^{j-s}-2q$, we see that $s$ is even. Write $2x=p^{s/2}r$, where $x\in \{0,\ldots, p^j-1\}$ and $r\in \{0,\ldots, p^{j-s/2}-1\}$. The congruence now becomes $r^2\equiv -2q\pmod{p^{j-s}}$, and in view of Lemma \ref{le:lem5}, it has at most two solutions contained in each interval of length $p^{j-s}$. Thus, the congruence has at most $2p^{j-s/2}/p^{j-s} = 2p^{s/2}$ solutions contained in each interval of length $p^j$. 
\end{enumerate}
Combining these cases, we see that
\[\begin{aligned}
\alpha_p &\leq \sum_{j\leq s}\left(j-\left\lceil\dfrac{j}{2}\right\rceil+1\right)\left\lceil\dfrac{n}{p^j}\right\rceil+\sum_{s < j\leq \log(2n^2+l)/\log p}2p^{s/2}\left\lceil\dfrac{n}{p^j}\right\rceil\\
&\leq \sum_{j\leq s}\left(j-\left\lceil\dfrac{j}{2}\right\rceil+1\right)\left(\dfrac{n}{p^j}+1\right)+\sum_{s < j\leq \log(2n^2+l)/\log p}2p^{s/2}\left(\dfrac{n}{p^j}+1\right)\\
&\leq \sum_{j\leq s}\left(\dfrac{j}{2}+1\right)\left(\dfrac{n}{p^j}+1\right)+2p^{s/2}\left(\dfrac{n(\frac{1}{p^{s+1}}-\frac{1}{p(2n^2+l)})}{1-\frac{1}{p}}+\dfrac{\log(2n^2+l)}{\log p}-s\right)\\
&=\dfrac{n}{2}\sum_{j\leq s}\dfrac{j+2}{p^j}+\dfrac{s(s+5)}{4}+\dfrac{2n}{p^{s/2}(p-1)}-\dfrac{2np^{s/2}}{(p-1)(2n^2+l)}+2p^{s/2}\left(\dfrac{\log(2n^2+l)}{\log p}-s\right)\\
&=\gamma_{l,p}(n).
\end{aligned}\]
\end{proof}

\begin{lemma}
\label{le:lem8}
For all positive integers $n$, we have $\sum_{n< p < 2n}\log p \leq n\log 4$.
\end{lemma}

\begin{proof}
	Observe that for primes $p$ with $n < p < 2n$, $p$ appears exactly once in the prime factorization of $\binom{2n}{n}$. Thus, by binomial theorem, we get $\prod_{n < p < 2n}p\leq \binom{2n}{n}\leq 4^n$, which is equivalent to the desired inequality.
\end{proof}

\begin{lemma}
\label{le:lem9}
Let $\pi(n)$ be the number of primes at most $n$. Then for all positive integers $n$, we have $\pi(n)\leq 2\log 4\frac{n}{\log n}+\sqrt{n}$.
\end{lemma}

\begin{proof}
	From \cite[p. 459, (22.4.2)]{hardyw}, we have 
	\[\pi(n)\leq n^{1-\lambda} + \dfrac{1}{(1-\lambda)\log n}\sum_{p\leq n}\log p\]
	for all $n\geq 1$ and $\lambda\in (0,1)$. Setting $\lambda = \frac{1}{2}$ and using $\prod_{p < n}p \leq 4^n$ (see \cite[Thm. 1.4]{tenen}), the desired inequality follows.
\end{proof}

\section{Proof of Theorem \ref{th:thm1} for odd $l\geq 3$}
We now give the proof of Theorem \ref{th:thm1} for odd $l\geq 3$ as follows. Suppose $n > \sqrt{l/2}$ and define the set $\mathcal{S} = \{p: p\text{ is an odd prime with }(-\frac{2l}{p})=1\}$. Then inequality (\ref{eq:eq3}) becomes 
\begin{align}
\label{eq:eq5}
\sum_{p\leq n} \beta_p\log p\leq \sum_{\substack{p \leq n\\p\in\mathcal{S}}} \dfrac{\alpha_p}{\lambda}\log p+\sum_{\substack{p \leq n\\p\notin\mathcal{S}}} \dfrac{\alpha_p}{\lambda}\log p+\sum_{n < p < 2n} \dfrac{\alpha_p}{\lambda}\log p.
\end{align}
Write $l = \prod_{i=1}^r p_i^{e_i}$, where $p_1,\ldots,p_r$ are odd primes. By Lemma \ref{le:lem6}, we have 
\begin{align}
\label{eq:eq6}
\sum_{\substack{p\leq n\\ p\notin\mathcal{S}}} \beta_p\log p \leq \sum_{\substack{p \leq n\\p\in\mathcal{S}}} \log(2n^2+l)+\dfrac{1}{\lambda}\sum_{i=1}^{r}\alpha_{p_i}\log p_i+\sum_{n < p < 2n} \dfrac{\alpha_p}{\lambda}\log p.
\end{align}
Observe that if $p > n$, then $\alpha_p \leq \lambda$ from (\ref{eq:eq4}), so applying Lemma \ref{le:lem8} in (\ref{eq:eq6}) yields
\begin{align}
\label{eq:eq7}
\sum_{\substack{p\leq n\\ p\notin\mathcal{S}}} \beta_p\log p \leq \log(2n^2+l)\sum_{\substack{p \leq n\\p\in\mathcal{S}}} 1+\dfrac{1}{\lambda}\sum_{i=1}^{r}\alpha_{p_i}\log p_i+n\log 4.
\end{align}
On the other hand, if $p \leq n$, then from (\ref{eq:eq2}) we have
\[\begin{aligned}
\beta_p &= \sum_{j\leq \log n/\log p}\left\lfloor\dfrac{n}{p^j}\right\rfloor\geq \sum_{j\leq \log n/\log p}\left(\dfrac{n}{p^j}-1\right)\geq n\sum_{j\leq \log n/\log p}\dfrac{1}{p^j}-\dfrac{\log n}{\log p}\\
&\geq n\left(\dfrac{1-\frac{1}{n}}{1-\frac{1}{p}}-1\right)-\frac{\log n}{\log p} = \dfrac{n-p}{p-1}-\dfrac{\log n}{\log p}\\
&\geq \dfrac{n-1}{p-1}-\dfrac{\log (2n^2+l)}{\log p}.
\end{aligned}\]
Thus, from (\ref{eq:eq7}) we get
\begin{align}
\label{eq:eq8}
(n-1)\sum_{\substack{p\leq n\\ p\notin\mathcal{S}}}\dfrac{\log p}{p-1}\leq \log(2n^2+l)\pi(n)+\dfrac{1}{\lambda}\sum_{i=1}^{r}\alpha_{p_i}\log p_i+n\log 4
\end{align}
and applying Lemma \ref{le:lem9}, we obtain 
\begin{align}
\sum_{\substack{p\leq n\\ p\notin\mathcal{S}}}\dfrac{\log p}{p-1}\leq \dfrac{\log(2n^2+l)}{n-1}\left(2\log 4\dfrac{n}{\log n}+\sqrt{n}\right)+\sum_{i=1}^{r}\dfrac{\alpha_{p_i}\log p_i}{\lambda (n-1)}+\dfrac{n\log 4}{n-1}.
\end{align}
Finally, using Lemma \ref{le:lem7}, we arrive at
\begin{align}
\label{eq:eq9}
\sum_{\substack{p\leq n\\ p\notin\mathcal{S}}}\dfrac{\log p}{p-1}\leq \dfrac{\log(2n^2+l)}{n-1}\left(2\log 4\dfrac{n}{\log n}+\sqrt{n}\right)+\dfrac{\log n}{\log(2n^2+l)}\sum_{i=1}^{r}\dfrac{\gamma_{l,p_i}(n)\log p_i}{n-1}+\dfrac{n\log 4}{n-1}.
\end{align}
As $n$ grows sufficiently large, the right-hand side of (\ref{eq:eq9}) approaches the limit
\begin{align}
\label{eq:eq10}
10\log 2+ \dfrac{1}{4}\sum_{i=1}^r\left(\dfrac{3p_i-2-p_i^{-e_i}}{(p_i-1)^2}-\dfrac{e_i+3}{p_i^{e_i}(p_i-1)}+\dfrac{4}{p_i^{e_i/2}(p_i-1)}\right)\log p_i,
\end{align}
while the left-hand side becomes unbounded since the prime number theorem implies that 
\[\sum_{p\leq n}\dfrac{\log p}{p-1} = \log n-\gamma + o(1)\]
where $\gamma$ is the Euler-Mascheroni constant (see \cite{tenen}). Thus, there exists a positive integer $N_l$ such that $\sum_{p\leq n, p\notin\mathcal{S}} (\log p)/(p-1)$ is greater than (\ref{eq:eq10}) for all $n\geq N_l$. Hence, we conclude that $P_{2,l}(n)$ is not a square for all $n\geq N_l$.

\section{Proof of Theorem \ref{th:thm1} for $l=1$}
We now give the proof of Theorem \ref{th:thm1} for $l=1$ as follows. Suppose $P_{2,1}(n)$ is a square for some $n\geq 1$ and let $p$ be a prime divisor of $P_{2,1}(n)$. By Lemma \ref{le:lem4}, we have $p < 2n$. Since $p$ divides $2k^2+1$ for some $1\leq k\leq n$, we have $(-\frac{2}{p})= 1$, so that $p\equiv 1,3\pmod{8}$. Thus, we have 
\[P_{2,1}(n) = \prod_{\substack{p < 2n\\ p\equiv 1,3\pmod{8}}} p^{\alpha_p}\]
and in view of $P_{2,1}(n) > (n!)^{\lambda}$, we see that
\begin{align}
\label{eq:eq12}
\sum_{p\leq n} \beta_p\log p\leq \dfrac{1}{\lambda}\sum_{\substack{p < 2n\\ p\equiv 1,3\pmod{8}}}\alpha_p\log p.
\end{align}
By Lemma \ref{le:lem6}, we have
\begin{align}
\label{eq:eq13}
\sum_{\substack{p\leq n\\ p\not\equiv 1,3\pmod{8}}} \beta_p\log p &\leq \sum_{\substack{p \leq n\\ p\equiv 1,3\pmod{8}}} \log(2n^2+1)+\sum_{n < p < 2n} \dfrac{\alpha_p}{\lambda}\log p.
\end{align}
Note that if $p > n$, then $\alpha_p \leq \lambda$ from (\ref{eq:eq4}) and if $p\leq n$, then
\[\beta_p \geq \dfrac{n-1}{p-1}-\dfrac{\log (2n^2+1)}{\log p}.\] 
Using the above bounds for $\alpha_p$ and $\beta_p$ and Lemmas \ref{le:lem8} and \ref{le:lem9} to (\ref{eq:eq13}), we obtain
\begin{align}
\label{eq:eq14}
\sum_{\substack{p\leq n\\ p\not\equiv 1,3\pmod{8}}}\dfrac{\log p}{p-1}\leq \dfrac{\log(2n^2+1)}{n-1}\left(2\log 4\dfrac{n}{\log n}+\sqrt{n}\right)+\dfrac{n\log 4}{n-1}.
\end{align}
As $n$ tends to infinity, the left-hand side of (\ref{eq:eq14}) becomes unbounded, while the right-hand side approaches to $10\log 2$. Since the smallest positive integer $n$ for which
\[\sum_{\substack{p\leq n\\ p\not\equiv 1,3\pmod{8}}}\dfrac{\log p}{p-1} > 10\log 2\]
is $n=706310$, we take $N_1:=706310$ so that $P_{2,1}(n)$ is not a square for $n\geq N_1$.

\section{Proof of Corollary \ref{co:cor2}}
We now apply Theorem \ref{th:thm1} to prove Corollary \ref{co:cor2}.

\begin{proof}
We know from the previous section that $P_{2,1}(n)$ is not a square for $n\geq N_1=706310$. Since $P_{2,1}(1)=3$ and $P_{2,1}(2)=27$ are not squares, it suffices to prove that $P_{2,1}(n)$ is not a square for $3\leq n\leq 706309$. We proceed as follows:
\begin{itemize}
	\item Since $2\cdot 3^2+1=19$ is a prime and the next value of $k > 3$ for which $19$ divides $2k^2+1$ is $k=19-3=16$, we see that $P_{2,1}(n)$ is not a square for $3\leq n\leq 15$.
	\item Since $2\cdot 6^2+1=73$ is a prime and the next value of $k > 6$ for which $73$ divides $2k^2+1$ is $k=73-6=67$, we see that $P_{2,1}(n)$ is not a square for $6\leq n\leq 66$.
	\item Since $2\cdot 21^2+1=883$ is a prime and the next value of $k > 21$ for which $883$ divides $2k^2+1$ is $k=883-21=862$, we see that $P_{2,1}(n)$ is not a square for $21\leq n\leq 861$.
	\item Since $2\cdot 597^2+1=712819$ is a prime and the next value of $k > 597$ for which $712819$ divides $2k^2+1$ is $k=712819-597=712222$, we see that $P_{2,1}(n)$ is not a square for $597\leq n\leq 712221$.
\end{itemize}
Hence, $P_{2,1}(n)$ is not a square for all $n\geq 1$.
\end{proof}

\section{Proof of Corollary \ref{co:cor3}}

We next prove Corollary \ref{co:cor3} using Theorem \ref{th:thm1}.

\begin{proof}
	We set $l=3$ with $r = 1, p_1=3$ and $e_1=1$. In this case, we have $\mathcal{S} = \{p\text{ prime}: p\equiv 1,5,7,11\pmod {24}\}$ and the limit (\ref{eq:eq10}) becomes $10\log 2+ \frac{1}{4}(1+\frac{2}{\sqrt{3}})\log 3\approx 7.523267$. Since the smallest positive integer $n$ for which
	\[\sum_{\substack{p\leq n\\ p\notin\mathcal{S}}}\dfrac{\log p}{p-1} > 10\log 2+ \dfrac{1}{4}\left(1+\dfrac{2}{\sqrt{3}}\right)\log 3\]
	is $n=N_3:=2189634$, we see that $P_{2,3}(n)$ is not a square for all $n\geq N_3$. Since $P_{2,3}(1)=5$ is not a square, it suffices to prove that $P_{2,3}(n)$ is not a square for $2\leq n\leq 2189633$. We proceed as follows:
	\begin{itemize}
		\item Since $2\cdot 2^2+3=11$ is a prime and the next value of $k > 2$ for which $11$ divides $2k^2+3$ is $k=11-2=9$, we see that $P_{2,3}(n)$ is not a square for $2\leq n\leq 8$.
		\item Since $2\cdot 8^2+3=131$ is a prime and the next value of $k > 8$ for which $131$ divides $2k^2+3$ is $k=131-8=123$, we see that $P_{2,3}(n)$ is not a square for $8\leq n\leq 122$.
		\item Since $2\cdot 37^2+3=2741$ is a prime and the next value of $k > 37$ for which $2741$ divides $2k^2+3$ is $k=2741-37=2704$, we see that $P_{2,3}(n)$ is not a square for $37\leq n\leq 2703$.
		\item Since $2\cdot 1048^2+3=2196611$ is a prime and the next value of $k > 1048$ for which $2196611$ divides $2k^2+3$ is $k=2196611-1048=2195563$, we see that $P_{2,3}(n)$ is not a square for $1048\leq n\leq 2199562$.
	\end{itemize}
	Hence, $P_{2,3}(n)$ is not a square for all $n\geq 1$.
\end{proof}

\section{Proof of Corollary \ref{co:cor4}}

We finally prove Corollary \ref{co:cor4} using Theorem \ref{th:thm1}.

\begin{proof}
	We set $l=7$ with $r = 1, p_1=7$ and $e_1=1$. In this case, we have $\mathcal{S} = \{p\text{ prime}: p\equiv 1,3,5,9,13,15,19,23,25,27,39,45\pmod {56}\}$ and the limit (\ref{eq:eq10}) becomes $10\log 2+ \frac{1}{4}(\frac{3}{7}+\frac{2}{3\sqrt{7}})\log 7\approx 7.262543$. Since the smallest positive integer $n$ for which
	\[\sum_{\substack{p\leq n\\ p\notin\mathcal{S}}}\dfrac{\log p}{p-1} > 10\log 2+ \dfrac{1}{4}\left(\dfrac{3}{7}+\dfrac{2}{3\sqrt{7}}\right)\log 7\]
	is $n=N_7:=2142500$, we see that $P_{2,7}(n)$ is not a square for all $n\geq N_7$. By direct calculation, we deduce that the only value of $n$ in $1\leq n\leq 5$ for which $P_{2,7}(n)$ is a square is $n=1$. Thus, it suffices to prove that $P_{2,7}(n)$ is not a square for $6\leq n\leq 2142499$. We proceed as follows:
	\begin{itemize}
		\item Since $2\cdot 6^2+7=79$ is a prime and the next value of $k > 6$ for which $79$ divides $2k^2+7$ is $k=79-6=73$, we see that $P_{2,7}(n)$ is not a square for $6\leq n\leq 72$.
		\item Since $2\cdot 15^2+7=457$ is a prime and the next value of $k > 15$ for which $457$ divides $2k^2+7$ is $k=457-15=432$, we see that $P_{2,7}(n)$ is not a square for $15\leq n\leq 431$.
		\item Since $2\cdot 39^2+7=3049$ is a prime and the next value of $k > 39$ for which $3049$ divides $2k^2+7$ is $k=3049-39=3010$, we see that $P_{2,7}(n)$ is not a square for $39\leq n\leq 3009$.
		\item Since $2\cdot 1041^2+7=2167369$ is a prime and the next value of $k > 1041$ for which $2167369$ divides $2k^2+7$ is $k=2167369-1041=2166328$, we see that $P_{2,7}(n)$ is not a square for $1041\leq n\leq 2166327$.
	\end{itemize}
	Hence, the only value of $n$ for which $P_{2,7}(n)$ is a square is $n=1$.
\end{proof}

%

\end{document}